\documentclass[10pt]{amsart}

\usepackage{amssymb}
\usepackage{amsmath} 
\usepackage[T1]{fontenc}
\usepackage{textcomp}
\usepackage[bookmarks=false,pdfstartview={FitBH},colorlinks={true},
pdfview={FitBH}]{hyperref}
\usepackage{orcidlink}
\usepackage{newpxtext,newpxmath}
\usepackage{enumitem}

\newtheorem{theorem}{Theorem}

\newtheorem{claim}[theorem]{Claim}

\theoremstyle{definition}
\newtheorem{definition}[theorem]{Definition}

\theoremstyle{remark}

\title{Ramsey Theory and Bounding in Arithmetic} 

\author[Peter Cholak]{Peter Cholak}
\address{University of Notre Dame}  
\email{Peter.Cholak.1@nd.edu}

\thanks{ORCID: \href{https://orcid.org/0000-0002-6547-5408}{\orcidlink{0000-0002-6547-5408} 0000-0002-6547-5408}. \\ \indent \textbf{Acknowledgements:} We dedicate this paper to Jeff Hirst on the occasion of his retirement. May he have many more happy years! While this author has been dealing with c.e.\ sets since his 1991 thesis \cite{MR2686153}, and $\mathrm{B}\Sigma^0_2$ since the work of Cholak, Jockusch, and Slaman \cite{CJS-2001} started in early 1997, this paper started as a talk for Hirst's retirement conference held in Vienna in August 2025 and grew from there. The author gave a talk on a similar topic in the summer of 2012 at the Midlands Logic Seminar, University of Birmingham. The author thanks Gavin Dooley, Damir Dzhafarov, Denis Hirschfeldt, Jeff Hirst, Richard Kaye, Leszek Ko{\l}odziejczyk, Ulrich Kohlenbach, Larry Moss, Ludovic Patey, Arno Pauly, and Keita Yokoyama for many useful conversations over the years related to this material. \\ \indent \textbf{Funding:} The author was partially supported by NSF-DMS-2502292 and the Erwin Schr{\"o}dinger International Institute for Mathematics and Physics.}

\begin{document}

\maketitle

\begin{abstract}

	We investigate the relationship between various versions of Ramsey's theorem and bounding schemes in a model $\mathcal{N}$ of a fragment of arithmetic $F$. Our primary objective is to recast and extend the seminal results of Hirst \cite{Hirst-1987} (see Theorem~\ref{hirst}) through the modern lens of Weihrauch reducibility. By extracting explicit Weihrauch reductions from classical proofs, we expose the uniform computational content underlying these combinatorial principles, yielding a deeper and more refined understanding of their reverse mathematical strength. 

	
	Our results, informally stated in our terminology and established inside $\mathcal{N}$, are as follows: The following are equivalent: $\mathrm{B}\Sigma^0_2$, the statement that the finite union of finite c.e.\ sets is finite, and the Infinite Pigeonhole Principle (see Theorem~\ref{equivBS2}). We also discuss the Weihrauch relations between these logically equivalent principles (see Section~\ref{swr}). The Infinite Pigeonhole Principle is Weihrauch reducible to $\mathrm{RT}^2_2$ (see Theorem~\ref{wr1}). There is also another principle logically equivalent to $\mathrm{B}\Sigma^0_2$ which is Weihrauch reducible to $\mathrm{SRT}^2_2$ (see Theorem~\ref{srt22}). We show that there is a principle which is equivalent to $\mathrm{B}\Sigma^0_3$ (see Theorem~\ref{BS3}) and Weihrauch reducible to $\mathrm{SRT}^2_{<\infty}$ (Theorem~\ref{srt2i}). We discuss some equivalences with $\mathrm{B}\Sigma^0_{n}$ (see Subsection~\ref{BSN}) and end with a problem Weihrauch reducible to $\mathrm{RT}^{n+1}_{2}$ (Subsection~\ref{RTN}). 
	
		
	Since we work within the model $\mathcal{N}$, many standard definitions must be adjusted. Due to the expository nature of this paper, these definitions are introduced throughout the text as needed. Reading the paper from start to finish will provide a better understanding of the ideas involved than focusing only on individual theorems.

\end{abstract}

\section{Introduction}

\subsection{The setup}

Our intended domain of discourse is a second-order model of arithmetic, $\mathcal{N}= (N,\mathcal{P},+,\times,<,0,1)$, where $\mathcal{P} \subseteq 2^N$ (a family of subsets of $N$). Our logical symbols are those of arithmetic and equality. We use a two-sorted logic with integers $x \in N$, integer variables $v$ which range over integers, sets $X\in \mathcal{P}$, set variables $Y$ which range over sets, and the membership relation $x \in X$. Our model need not be standard (i.e., $N$ may not be $\omega$), and $\mathcal{P}$ is most likely a proper subset of $2^N$. Elements of $N$ will be called integers, and elements $X\in \mathcal{P}$ will be called sets. 

We say that a formula is $\Sigma^0_0$, $\Pi^0_0$, or $\Delta^0_0$ if all of its quantifiers range over integer variables and are bounded (i.e., $\forall v < t$ or $\exists v < t$). Number and set parameters are allowed in our formulas; for simplicity, we suppress these parameters in our notation hereafter. We also establish the convention that $\theta$ always denotes a $\Delta^0_0$ formula. A formula is $\Sigma^0_{n+1}$ if and only if (iff) it consists of a block of existential quantifiers followed by a formula in $\Pi^0_n$, and a formula is $\Pi^0_{n+1}$ iff it consists of a block of universal quantifiers followed by a formula in $\Sigma^0_n$. A $\Delta^0_{n+1}$ formula is one equivalent in $\mathcal{N}$ to both a $\Sigma^0_{n+1}$ and a $\Pi^0_{n+1}$ formula. At this point, we make no claims about what formulas might or might not be equivalent in $\mathcal{N}$ to $\Delta^0_0$, $\Sigma^0_n$, or $\Pi^0_n$ formulas. An arithmetic formula is one which is $\Sigma^0_{n}$ or $\Pi^0_{n}$ for some $n$. 

Induction for a property or set $P \subseteq \mathbb{N}$ is the statement that if $P(0)$ holds, and $P(n)$ implies $P(n+1)$ for all $n$, then $P(n)$ holds for all $n$. If $\mathcal{N}$ is standard, then induction holds for all possible $P$. If $\Gamma$ is a collection of subsets of $\mathbb{N}$, we write $\mathrm{I}\Gamma$ to denote the scheme stating that induction holds for all sets in $\Gamma$. Typical choices for $\Gamma$ are the classes of $\Delta^0_0$, $\Sigma^0_n$, or $\Pi^0_n$ definable sets. For example, $\mathrm{I}\Delta^0_0$ is the induction scheme for the sets $\{ x \in \mathbb{N} \mid \mathcal{N} \models \theta(x)\}$, where $\theta$ is a $\Delta^0_0$ formula (where again, set parameters are allowed in $\theta$). Arithmetic induction is induction for all arithmetically definable sets. 

We need $\mathcal{N}$ to realize some fragment $F$ of true second-order arithmetic, i.e., the theory of the structure $(\omega, 2^{\omega}, +, \times, <, 0, 1)$. $F$ must include $P^-$, Peano Arithmetic without induction, and $\mathrm{I}\Delta^0_0$. A set $X$ in $\mathcal{N}$ is finite iff there is an integer $n \in \mathbb{N}$ such that for all $x \in X$, $x \leq n$. In Section~\ref{ct}, we need to code finite objects as integers via some sort of G\"{o}del coding. For this purpose, we require exponentiation to be total. The details of showing that exponentiation is definable in $P^- + \mathrm{I}\Delta^0_0$ are nontrivial, and we encourage the interested reader to explore Section~3 of Chapter~V in H\'{a}jek and Pudl\'{a}k \cite{Hajek.Pudlak:93}. However, $P^- + \mathrm{I}\Delta^0_0$ is not enough to show that exponentiation is total (i.e., that for all $x$ and $y$, there is a $z$ such that $x^y=z$). $F$ must include at least the assertion that exponentiation is total, or something stronger like $\mathrm{I}\Sigma^0_1$ ($\mathrm{I}\Sigma^0_1$ can show that exponentiation is total). 

Informally, $\mathcal{P}$ should be closed under Turing reducibility. We formalize this via $\Delta^0_1$-Comprehension ($\Delta^0_1\text{-CA}$). $\Delta^0_1\text{-CA}$ states that if a formula $\varphi(x)$ is equivalent in $\mathcal{N}$ to both a $\Sigma^0_1$ formula and a $\Pi^0_1$ formula, then $\{ x \in \mathbb{N} \mid \mathcal{N} \models \varphi(x) \} \in \mathcal{P}$. Thus, $F$ must include $\Delta^0_1\text{-CA}$, implying that at a minimum, $\mathcal{P}$ must contain the computable sets. We should point out that throughout this paper, we discuss definable sets (i.e., sets of the form $\{ x \in \mathbb{N} \mid \mathcal{N} \models \varphi(x) \}$). However, unless $\varphi$ is equivalent to a $\Delta^0_1$ formula, these sets need not reside in $\mathcal{P}$. Note that there can be set parameters in $\varphi$ drawn from $\mathcal{P}$. A definable partial function $f$ is one whose graph $G = \{ (x,y) \mid y = f(x) \}$ is a definable relation. We also assume if $(x,y) \in G$ and $(x,z) \in G$, then $y=z$. Hereafter, $\mathcal{N}$ must be a model of $F$; this theory $F$ is also called $\mathrm{RCA}_0^*$. 

Our reference for all things related to first-order arithmetic is Kaye \cite{Kaye-1991}, though another excellent choice is H\'ajek and Pudl\'ak \cite{Hajek.Pudlak:93}. For second-order arithmetic, we cite Simpson \cite{Simpson-2009} or Dzhafarov and Mummert \cite{MR4472209}. 

\subsection{Bounding}

The bounding scheme for a property $P$ states that the following statement $\Theta(n)$ holds for all $n$:
\begin{equation*}\begin{split}
	\forall i < n \exists & b_i \bigg[  {P}(i,b_i) \bigg]\implies \\ \exists b \forall i < n \exists  &b_i < b \bigg[ {P}(i,b_i) \bigg]
\end{split}
\end{equation*}
Clearly, $\Theta(0)$ holds trivially, and if $\Theta(n)$ holds with witness $b$, then $\max \{ b, b_n \}$ witnesses that $\Theta(n+1)$ holds. Thus, a sufficient level of induction implies bounding. It is instructive to state bounding explicitly for the $\Sigma^0_2$ formula $\exists a \forall c \theta(i,b_i,a,c)$: the following holds for all $n$:
\begin{equation}\label{BS2}\begin{split}
	\forall i < n \exists & b_i \bigg[  \exists a \forall c \theta(i,b_i,a,c) \bigg]\implies \\ \exists b \forall i < n \exists  &b_i < b \bigg[\exists a \forall c \theta(i,b_i,a,c) \bigg]
\end{split}
\end{equation}
The bounding scheme $\mathrm{B}\Sigma^0_n$ is the bounding scheme restricted to $\Sigma^0_n$ formulas. For any $n$, $\mathrm{B}\Sigma^0_n$ can be used to show that every formula is equivalent in $\mathcal{N}$ to one in Prenex Normal Form, where the bounded quantifiers are in the interior of the formula. This fact will be needed twice, once each in the proof of Theorem~\ref{useBn} and Lemma~\ref{lastlemma}. 
	
Kirby and Paris \cite{Kirby.Paris:77} showed that for all $n$, in $\mathcal{N}$, $\mathrm{I}\Sigma^0_{n+1}$ strictly implies $\mathrm{B}\Sigma^0_{n}$, which in turn strictly implies $\mathrm{I}\Sigma^0_{n}$. $\mathrm{B}\Sigma^0_n$, for $n\geq 2$, is equivalent to a $\Pi^0_{n+2}$ sentence (see \cite{Hajek.Pudlak:93}). 
  
\subsection{Ramsey's Theorem}

The \emph{Infinite Pigeonhole Principle} states that if one has infinitely many pigeons and computably puts them into finitely many pigeonholes, at least one pigeonhole will have infinitely many pigeons. Formally, the Infinite $\Gamma$-Pigeonhole Principle asserts that if $C$ is a $\Gamma$-definable (in $\mathcal{N}$) partial coloring  ($C$ is a function whose graph is $\Gamma$-definable) from $\mathbb{N}$ to some integer $k$ and $X$ is an infinite subset of the domain of $C$, then there is an $x<k$ where $C^{-1}(x) \cap X$ is an infinite set in $\mathcal{N}$. We will consider the Infinite Pigeonhole Principle as the Infinite $\Sigma^0_1$-Pigeonhole Principle. We denote this principle by $\mathrm{RT}^1_{< \infty}$. Readers might expect us to consider the Infinite $\Delta^0_1$-Pigeonhole Principle, but we shall see that a solution to the Infinite $\Sigma^0_1$-Pigeonhole Principle implies one for the Infinite $\Delta^0_1$-Pigeonhole Principle. This will be discussed further in Sections~\ref{ct} and \ref{wr}. The Infinite Pigeonhole Principle is conceptually distinct from the finite Pigeonhole Principle, which states that if $m>n$ are integers and $f$ is a function from $m$ to $n$, then there is an $x<n$ where $|f^{-1}(x)| \geq 2$. 

Let $[X]^n = \{ Y \subset X \mid |Y| = n \}$, where $X$ is $\mathcal{N}$-infinite. $\mathrm{RT}^n_k$ states that for every $k$-coloring $C$ of $[X]^n$, there is a homogeneous set $H$ such that $C$ is constant on $[H]^n$. $\mathrm{RT}^n_{<\infty}$ is $\mathrm{RT}^n_k$ for all $k$. A $k$-coloring $C(x,s)$ is \emph{stable} if $\lim_s C(x,s)$ exists. $\mathrm{SRT}^2_k$ is $\mathrm{RT}^2_k$ restricted to stable colorings. 

\subsection{Hirst's Theorems}

The following results of Hirst are our main interest. These results were originally established in his 1987 doctoral dissertation \cite{Hirst-1987}. 

\begin{theorem}[Hirst \cite{Hirst-1987}]\label{hirst} Over $\mathrm{RCA}_0$ ($P^- + \mathrm{I}\Sigma^0_1 + \Delta^0_1\text{-CA}$):
\begin{enumerate}[label=(\alph*)]
	\item $\mathrm{RT}^1_{<\infty}$ is equivalent to $\mathrm{B}\Sigma^0_2$. 
 	\item $\mathrm{RT}^2_2$ implies $\mathrm{B}\Sigma^0_2$.
 	\item $\mathrm{RT}^2_{<\infty}$ implies $\mathrm{B}\Sigma^0_3$.
\end{enumerate}
\end{theorem}

These results have had a lasting impact on reverse mathematics and are highly deserving of reflection. Our goal is to reprove and recast these results using Weihrauch reducibilities. We will see that there are still some interesting unexplored issues and subtleties about these results. While we expose some of these nuances, others remain open for future study. First, some basic computability theory inside $\mathcal{N}$ and a brief discussion of Weihrauch reductions are needed. 

\subsection{Computability Theory}\label{ct}

We are going to work inside $\mathcal{N}$. Kleene's $T$-predicate, $T(e,x,s)$, holds iff $s$ is the computation history of a halting run of the $e$th Turing machine with input $x$. Using course-of-values recursion, $s$ codes the sequence of configurations of the $e$th Turing machine with input $x$ for each stage in the run. This relies on the fact that exponentiation is total and G\"odel coding is available (for the reader wishing to learn more about course-of-values recursion and G\"odel coding in our fragment $F$, we direct them to H\'ajek and Pudl\'ak \cite{Hajek.Pudlak:93}). Kleene \cite{MR1513071} showed that the $T$-predicate, $T(e,x,s)$, is a primitive recursive relation. By Bennett \cite{MR2613641} (see Chapter I, Section 3(c) of \cite{Hajek.Pudlak:93}), this relation is $\Delta^0_0$-definable in $\mathcal{N}$. If we have $\mathcal{N} \models \mathrm{I}\Sigma^0_1$, then in $\mathcal{N}$ all the primitive recursive functions are total, but that is not necessary here.  

$K = \{ e \mid \exists s T(e,e,s) \}$ is a $\Sigma^0_1$ set. It is also $\Sigma^0_1$-complete, i.e., for all $\Sigma^0_1$ formulas $\chi(x)$, there is a computable total one-to-one function $f$ such that $\chi(x)$ iff $f(x) \in K$. Let us denote $x \in W_{e,s}$ iff $\exists t< s T(e,x,t)$. This is equivalent to the standard definition of a c.e.\ set. Hence ``$x\in W_{e,s}$'' is equivalent to a $\Delta^0_0$ formula $\varphi(x,s)$. Every $\Delta^0_0$ formula $\theta(x,s)$ is equivalent to ``$x\in W_{e,s}$'' for some $e$. The c.e.\ sets are the $\Sigma^0_1$-definable sets in $\mathcal{N}$. For a set to be computable, both it and its complement need to be $\Sigma^0_1$-definable sets. So this set is $\Delta^0_1$. We will frequently equate computable objects with $\Delta^0_1$-definable objects. We consider a Turing functional $\Theta(X)$ as a $\Delta^0_1$-definable formula with a free variable $X$. 
	
One theorem we will need several times is that every infinite c.e.\ set uniformly contains an infinite computable subset. For example, consider a $\Sigma^0_1$-definable (in $\mathcal{N}$) partial coloring $C$ from $\mathbb{N}$ to some integer $n$. The domain of $C$ is a c.e.\ set; if its domain contains an $\mathcal{N}$-infinite subset, then it contains a computable $\mathcal{N}$-infinite subset $X$. We can always assume we are working with this computable subset of the domain when the domain is infinite. This subset can be obtained uniformly from the coloring. 
	
The index set $\mathrm{FIN} = \{ e \mid W_e \text{ is finite} \} = \{ e \mid \exists b \forall s > b (W_{e,b} = W_{e,s})\}$ is $\Sigma^0_2$-complete. For all $\theta$, there is a computable total one-to-one function $f$ such that:
\begin{equation}\label{FU}\begin{split}
	\exists  b_i  \exists a \forall c & \theta(i,b_i,a,c) \text{ iff} \\ 
		f(i) &\in \mathrm{FIN} \text{ iff}\\
		\exists b_i \forall s>b_i &(W_{f(i),b_i} = W_{f(i),s}).
\end{split}\end{equation}
$\mathrm{COF} = \{ e \mid W_e \text{ is cofinite} \} = \{ e \mid \exists k \forall x> k \exists s (x \in W_{e,s})\}$ is $\Sigma^0_3$-complete. Unfortunately, there are no similar ``nice'' index sets available for $\Sigma_n$ when $n>3$. 

A good computability reference is Soare \cite{Soare:87}, but there are plenty of newer references available. For working carefully in $\mathcal{N}$, a good start is Simpson \cite{Simpson-2009}. 

\subsection{Problems}

We will consider the $\mathcal{N}$-arithmetical formula $\mathcal{A}(I,S)$ asserting that $S$ is a solution of some instance $I$ of some fixed solvable problem $\mathcal{A}$. There are no parameters in $\mathcal{A}$ other than the free variables $I$ and $S$. A primary example is ${\mathrm{RT}}^2_2(C,H)$, which says $C$ is a $2$-coloring of all pairs in some $\mathcal{N}$-infinite set $X \subseteq \mathbb{N}$ and $H \subseteq X$ is a homogeneous $\mathcal{N}$-infinite set. So $C$ is the instance and $H$ is the solution. In what follows, we recast some of our relevant arithmetic statements as problems. We will have to be careful to identify the instances and solutions of our problems. 

Given a problem $\mathcal{A}(I,S)$, we say $\mathcal{N}$ \emph{realizes} $\mathcal{A}$ iff $\mathcal{N}$ realizes that for all $I$ there is an $S$ such that $\mathcal{A}(I,S)$. We assume that if $I$ is not an instance of the problem $\mathcal{A}$ (say, not a coloring, etc.), then $S = \emptyset$. This is a reasonable assumption since determining whether $I$ is an instance of any of our problems is always arithmetic.

\subsection{Weihrauch Reducibility over $\mathcal{N}$}\label{wr}
	
\begin{definition}\label{equiv}
   We say $\mathcal{A}_0(I_0,S_0)$ is \emph{(strongly) Weihrauch reducible over $\mathcal{N}$} to $\mathcal{A}_1(I_1,S_1)$, written $\mathcal{A}_0 \leq_W \mathcal{A}_1$ or $\mathcal{A}_0 \leq_{sW} \mathcal{A}_1$ iff there are two Turing functionals $\Phi$ and $\Psi$ such that:
     \begin{itemize}
         \item $\mathcal{N}$ realizes $\forall I_0 \forall S_1 ( \mathcal{A}_1(\Phi(I_0),S_1)) \Rightarrow \mathcal{A}_0(I_0,\Psi(I_0, S_1)).$
         \item $\mathcal{N}$ realizes $\mathcal{A}_1$. 
         \item If the backward functional depends only on $S_1$ (i.e., we can write $\Psi(S_1)$ instead of $\Psi(I_0, S_1)$), then the reduction is strong. 
         \item Both $\Phi(I_0)$ and $\Psi(I_0,S_1)$ are $\Delta^0_1$-definable partial functions which are total on all reasonable oracles (sets), $I_0$ and $S_1$, from $\mathcal{P}$.
     \end{itemize}
\end{definition}

A Weihrauch reduction is the gold standard in showing one problem uniformly computes (or implies) another. If $\mathcal{N}$ is the standard model, this definition is the standard definition. This reduction, though not explicitly named, was in use in this fashion by Hirst and others long before it was named. For example, the reduction used to show that $\mathrm{SRT}^2_2$ follows from $\mathrm{RT}^2_2$ is a strong Weihrauch reduction with the needed functionals being the identity: $\mathrm{SRT}^2_2 \leq_{sW} \mathrm{RT}^2_2$. Other examples will follow. 

We are going to use Weihrauch reducibility as a replacement for a proof from $F$. If $\mathcal{A}_0 \leq_W \mathcal{A}_1$ in $\mathcal{N}$, then $\mathcal{N}$ realizes $\mathcal{A}_0$. This implies that $F$, our fragment of arithmetic, proves that $\mathcal{A}_1$ implies $\mathcal{A}_0$. If $\mathcal{N}$ does not realize $\mathcal{A}_1$, the Weihrauch reduction (and resulting proof) might be meaningless. Hence, the second clause in our above definition is necessary.  

There is similar work in Reitzes \cite{MR4494715} and Dzhafarov, Hirschfeldt, and Reitzes \cite{MR4520553}, where a proof of $\mathcal{A}_1 \Rightarrow \mathcal{A}_0$ in $\mathrm{RCA}_0$ is used to provide a winning strategy for a certain related game. There is one important difference between our definition and the ones considered in \cite{MR4494715} and \cite{MR4520553} in that we assume $\mathcal{N}$ realizes $\mathcal{A}_1$. We can weaken that statement to $\forall I_0 \exists S_1 ( \mathcal{A}_1(\Phi(I_0),S_1))$ if necessary. Concretely, in \cite{MR4520553}, they call the above reduction $\leq_W^{\mathrm{RCA}_0^*}$ and discuss $\mathrm{B}\Sigma^0_2$ in Sections 6 and 7.  

The Infinite $\Delta^0_1$-Pigeonhole Principle is strongly Weihrauch reducible to the Infinite $\Sigma^0_1$-Pigeonhole Principle via the identity functionals. Every $\Sigma^0_1$ partial coloring has a $\Sigma^0_1$ domain. Every infinite $\Sigma^0_1$ set has an infinite $\Delta^0_1$ set $X$. Hence, the Infinite $\Sigma^0_1$-Pigeonhole Principle is strongly Weihrauch reducible to the Infinite $\Delta^0_1$-Pigeonhole Principle, where $\Phi$ restricts the given instance $C_0$ to the set $X$ and $\Psi$ is the identity functional (all of the above is in $\mathcal{N}$).

Given a problem $\mathcal{A}$, if there is a Turing functional $\Xi$ such that $\mathcal{N}$ realizes $\forall I \mathcal{A}(I,\Xi(I))$, then $\mathcal{A}$ is Weihrauch below every problem. We will call $\mathcal{A}$ \emph{Weihrauch trivial} or just trivial in $\mathcal{N}$. If a problem $\mathcal{A}_0$ is Weihrauch below a trivial problem $\mathcal{A}$, it is also trivial, witnessed by $\Psi(I_0,\Xi(\Phi(I_0)))$. 

One of our main themes is to view Weihrauch reductions as a proof-theoretic tool. While we do not systematically investigate the general Weihrauch degree structure over $\mathcal{N}$, in Section~\ref{swr} we explore the Weihrauch reductions between a number of principles or problems logically equivalent to $\mathrm{B}\Sigma^0_2$. There are other places within the paper where it is possible to make the same diversion, but in those cases we have found restraint. Good references for the Weihrauch reducibilities (over the standard model) and related material are Hirschfeldt and Jockusch \cite{MR3518779} and Dzhafarov and Mummert \cite{MR4472209}.

\section{Equivalences of $\mathrm{B}\Sigma^0_2$}

The following proves the first result of Hirst, Theorem~\ref{hirst}(a). Our first clause (a) is discussed in almost this form in \cite{MR4472209}, see Definition 3.3.1, and Section 2 of Frittaion and Marcone \cite{MR2997030}. We will discuss the Weihrauch relations (over $\mathcal{N}$) between these principles and another principle in Section~\ref{swr}.

\begin{theorem}\label{equivBS2}
	In $\mathcal{N}$, the following are equivalent:
	\begin{enumerate}[label=(\alph*)]
		\item Every finite union of finite c.e.\ sets is finite.
		\item $\mathrm{B}\Sigma^0_2$.
		\item The Infinite Pigeonhole Principle. 
	\end{enumerate}
\end{theorem}

\begin{proof}
	First, we note that the statement ``the finite union of finite c.e.\ sets is finite'' follows from $\mathrm{B}\Sigma^0_2$: Fix $n$ and a computable total function $f$. Then 	
	\begin{equation}\label{FU2}\begin{split}
		\forall i < n \exists b_i \bigg[ \forall t > b_i(W_{f(i),b_i}& =W_{f(i),t}) \bigg]\implies \\ \exists b \forall i < n  \bigg[ \forall t > b (W_{f(i),b}& =W_{f(i),t})\bigg]
	\end{split}
	\end{equation}
	is an occurrence of $\mathrm{B}\Sigma^0_2$ (see Equation~\ref{BS2}). So (b) implies (a). The following implications show that (a) implies (b):
	\begin{equation}\label{FU3}\begin{split}
		\forall i < n \exists & b_i \bigg[  \exists a \forall c \theta(i,b_i,a,c) \bigg]\implies \\
		\forall i < n  \exists b_i \forall s>b_i &\bigg[ W_{f(i),b_i} = W_{f(i),s} \bigg]\implies \\ \exists b \forall i < n  \bigg[ \forall t> b &(W_{f(i),b} =W_{f(i),t})\bigg] \implies
		\\ \exists b \forall i < n \exists  &b_i < b \bigg[\exists a \forall c \theta(i,b_i,a,c) \bigg]
	\end{split}
	\end{equation} 
	The first clause is the hypothesis of Equation~\ref{BS2}. The last is the conclusion. The second clause follows from Equation~\ref{FU}. The third clause follows from (a) (see Equation~\ref{FU2}). The last implication follows from Equation~\ref{FU}.
	
	The contrapositive of the Infinite Pigeonhole Principle asserts that if one has finitely many pigeonholes, each containing finitely many pigeons, then the total number of pigeons is finite. Let $C$ be a $\Sigma^0_1$-definable partial coloring into $n$. Consider the c.e.\ sets $W_{f(i)} = \{x \mid C(x)=i\}$. If every $W_{f(i)}$ is finite for $i<n$, then (a) implies that the domain of $C$ is finite. A finite collection of c.e.\ sets $W_{f(i)}$ for $i < n$ induces a $\Sigma^0_1$ partial coloring: $C(x)=i$ iff there is an $s$ such that $i$ is the least $e$ such that $x \in W_{f(e),s}$ and for all $t<s$ and all $i < n$, $x \not\in W_{f(i),t}$. If the $W_{f(i)}$ are finite, then (c) implies the domain of $C$ is finite, and hence the union of the $W_{f(i)}$ is also finite. Hence, (a) and (c) are equivalent. 	
\end{proof}

Using compactness, one can construct a (computable) first-order model $\mathcal{M}$ of $P^- + \mathrm{I}\Sigma^0_1 + \neg \mathrm{B}\Sigma^0_2$. To make this a model of our required second-order theory, just add the $\mathcal{M}$-computable sets. In $\mathcal{M}$, the failure of $\mathrm{B}\Sigma^0_2$ can always be witnessed by a sequence of $\mathcal{M}$-finite, downward closed (under $<_\mathcal{M}$) $\mathcal{M}$-c.e.\ sets $\{ W_{f(e)} \mid e < a\}$, where $a$ is a nonstandard integer of $\mathcal{M}$ and $f$ is $\mathcal{M}$-computable, whose union is $|\mathcal{M}|$, all integers in $\mathcal{M}$. 

In Section~3 of an unpublished paper by Groszek and Slaman \cite{Groszek.Slaman:94}, there is a construction of a model where $\mathrm{B}\Sigma^0_2$ fails but $\mathrm{I}\Sigma^0_1$ holds. One can modify that construction to get an explicit $f$ and $a$ witnessing the failure of $\mathrm{B}\Sigma^0_2$ as above. Similar models are discussed in Belanger, Chong, Wang, and Yang \cite{MR4328699} and in Haken's thesis \cite{MR3295326}. 

\section{Weihrauch reductions from $\mathrm{RT}^2_2$ and $\mathrm{SRT}^2_2$} 

In this section and Section~\ref{hirst3}, we show the two remaining results of Hirst using Weihrauch reducibilities inside $\mathcal{N}$. We are going to consider some of the above arithmetic formulas as problems. The following theorem also appears in Hirschfeldt and Jockusch \cite{MR3518779} as Theorem~2.10(8). It proves the second theorem of Hirst, Theorem~\ref{hirst}(b).  

\begin{theorem}\label{wr1}
  The Infinite Pigeonhole Principle is (strongly) Weihrauch reducible over $\mathcal{N}$ to $\mathrm{RT}^2_2$ (or even $\mathrm{RT}^n_2$, which will be discussed further in Section~\ref{last}). 
\end{theorem}

\begin{proof}
	Colorings whose domain includes an infinite computable set $X$ are our instances of the Infinite Pigeonhole Principle. Let $C_0$ be such a finite coloring. We define the coloring $\Phi(C_0)$ on $[X]^2$ by setting $\Phi(C_0)(x,y) = 1$ if $C_0(x) = C_0(y)$, and $0$ otherwise (if we are coloring $n$-tuples, just ignore all but the first $2$ elements to get a coloring of $[X]^n$). Let $H_1 \subseteq X$ be homogeneous for $\Phi(C_0)$. Since there are only finitely many colors available for $C_0$, $H_1$ cannot be homogeneous with color $0$ (which would require all elements of $H_1$ to have distinct colors under $C_0$). So all integers in $H_1$ have the same color under $C_0$. 
	
	We have a choice regarding the representation of the solution: do we return an infinite subset of a single color class, or do we return the index of the color class itself? In the former case, we can set $\Psi(H_1) = H_1$, yielding a strong Weihrauch reduction. In the latter case, we require the original coloring $C_0$ to identify the color: choosing any $h \in H_1$, we set $\Psi(C_0, H_1) = C_0(h)$, which gives a standard Weihrauch reduction.
\end{proof}

Let's assume $\mathrm{RT}^1_{<\infty}$ always has a solution in $\mathcal{N}$. Given a hole $h$, we can enumerate the set of pigeons entering the hole, $\{x \mid C_0(x)=h\}$. This set is either finite or infinite. In the latter case, there is an infinite computable subset of $\{x \mid C_0(x)=h\}$. So if we know which hole is infinite, we can find our solution computably. However, a diagonalization argument shows that there is no Turing functional $\Xi$ such that $\Xi(C_0)$ is a solution to $C_0$. So $\mathrm{RT}^1_{<\infty}$ is not trivial but does have computable solutions. By Specker \cite{Specker:71}, there is an instance of $\mathrm{RT}^2_2$ without any computable solutions. So $\mathrm{RT}^2_2 \nleq_W \mathrm{RT}^1_{<\infty}$. 
 
That leaves us with stable colorings and $\mathrm{SRT}^2_2$. The above reduction does not work since the coloring $\Phi(C_0)$ is not stable when more than one hole has infinitely many pigeons. But something more is true. $\mathrm{RT}^1_{k}$ is the Infinite Pigeonhole Principle when there are only $k$ pigeonholes. By Hirschfeldt and Jockusch \cite{MR3518779} and Brattka and Rakotoniaina \cite{MR3743611}, we know $\mathrm{RT}^1_{3} \nleq_{W} \mathrm{SRT}^2_2$ (see Theorem~2.10(4) of \cite{MR3518779}). This is a diagonalization argument exploiting the uniform nature of Weihrauch reducibility and the number of colors available. So $\mathrm{RT}^1_{<\infty} \not\leq_{W} \mathrm{SRT}^2_2$. Their proof is done in the standard model but also holds in $\mathcal{N}$. For more details, see the next subsection. 
  
Nevertheless, Hirst's result was improved to $\mathrm{SRT}^2_2$ in Cholak, Jockusch, and Slaman \cite{CJS-2001} (see Lemma~10.6). Our goal is to abstract a Weihrauch reduction from that proof. Consider the statement logically equivalent to the Infinite Pigeonhole Principle: if we color infinitely many pigeons using a coloring $C_0$ such that each color is used only finitely often (the instance), then there are infinitely many colors used (the solution). This is a sort of Rainbow Pigeonhole, and we will call it ``Rainbow Pigeonhole''. Again, in this case, we can assume that the domain of our coloring contains a computable set $X$ and just work within $X$.   
  
\begin{theorem}\label{srt22}
	Rainbow Pigeonhole is Weihrauch reducible over $\mathcal{N}$ to $\mathrm{SRT}^2_2$. 
\end{theorem}

\begin{proof}
	Let $C_0$ be a coloring such that each color is used only finitely often. Like above consider the coloring $\Phi(C_0)(x,y)=1$ iff $C_0(x)=C_0(y)$, and $0$ otherwise. For all $x$, $\lim_y \Phi(C_0)(x,y)=0$.  Hence $\Phi(C_0)$ is a stable $2$-coloring of pairs. Let $H_1$ be homogeneous for $\Phi(C_0)$. Now $H_1$ must have color $0$, and from $H_1$ and $C_0$, one can compute infinitely many colors.
\end{proof}

\section{Weihrauch relations over $\mathcal{N}$ among problems equivalent to $\mathrm{B}\Sigma^0_2$}\label{swr}

This section can be omitted on a first reading of the paper. Here we explore the Weihrauch reductions between the various statements logically equivalent to $\mathrm{B}\Sigma^0_2$ discussed in the previous section. This work drives home the fact that the Weihrauch degrees need not be closed under logical equivalences. We should note that similar work has been done comparing Ramsey-theoretic problems equivalent to $\mathrm{I}\Sigma^0_2$; see Davis, Hirschfeldt, Hirst, Pardo, Pauly, and Yokoyama \cite{doi:10.3233/COM-180244} and Pauly, Pradic, and Sold\`{a} \cite{doi:10.3233/COM-230437}. 

If an instance of Rainbow Pigeonhole has a solution in $\mathcal{N}$, we can enumerate the colors $\{i \mid \exists x (C_0(x)=i)\}$, and this infinite c.e.\ set must contain an infinite computable set. The context of Theorem~\ref{srt22} is just to show that Rainbow Pigeonhole is realized in $\mathcal{N}$ when $\mathcal{N}$ realizes $\mathrm{SRT}^2_2$. If Rainbow Pigeonhole is realized, then it is a trivial problem. If $\mathcal{N}$ realizes $\mathrm{RT}^1_{< \infty}$, then Rainbow Pigeonhole is also realized (as they are logically equivalent). This, and the fact that $\mathrm{RT}^1_{<\infty} \not\leq_{W} \mathrm{SRT}^2_2$, implies that $\text{Rainbow Pigeonhole} <_{W} \mathrm{RT}^1_{< \infty}$ even though they are logically equivalent.  
  
Let's consider the statement ``the finite union of finite c.e.\ sets is finite'' as a problem. The instance is a collection of finite c.e.\ sets $\{ W_{f(e)} \mid e < n\}$. Instances of this problem do not include a bound on the size of each individual set. If the size were included, we could just add them up (addition is total in $\mathcal{N}$) to get a bound on the union. A solution is a bound on the size of the union (or if you want, an infinite set of bounds). A diagonalization argument, using the recursion theorem, shows that there is no Turing functional $\Xi$ such that $\Xi(\{ W_{f(e)} \mid e < n\})$ is a solution. So this problem, if it has a solution, is not trivial. Hence, this problem is also strictly Weihrauch above Rainbow Pigeonhole.

The above holds for $\mathrm{B}\Sigma^0_2$ as a problem. The proof of Theorem~\ref{equivBS2} actually shows that ``the finite union of finite c.e.\ sets is finite'' and $\mathrm{B}\Sigma^0_2$ are strongly Weihrauch equivalent. 

Viewed as a problem, ``the finite union of finite c.e.\ sets is finite'' is strongly Weihrauch reducible to $\mathrm{RT}^1_{< \infty}$: Let $\Phi(\{ W_{f(e)} \mid e < n\})(s+1) = s+1$ iff there is an $e<n$ such that $W_{f(e),s} \neq W_{f(e),s+1}$ and $\Phi(\{ W_{f(e)} \mid e < n\})(s)$ otherwise. However, this reduction is strict. Working inside $\mathcal{N}$, assume for a contradiction that $\mathrm{RT}^1_{< \infty}$ is Weihrauch reducible to ``the finite union of finite c.e.\ sets is finite'' via $\Phi$ and $\Psi$. We will computably build an instance $C$ of $\mathrm{RT}^1_{< \infty}$, with more details in a few sentences. Using the recursion theorem, we can assume that we know $\Phi(C)$. We can assume $\Phi(C)$ is a computable finite set of finite c.e.\ sets $\{ W_{f(e)} \mid e < n \}$. At stage $s$, let $x_s$ be the least upper bound of $\bigcup \{ W_{f(e),s} \mid e < n \}$. If our Weihrauch reduction is correct, $\lim_s x_s$ exists. At some stage $s$, $[x_s, \infty)$ becomes a solution to $\Phi(C)$. At some sufficiently large stage $s$, $\Psi_s(C,[x_s, \infty))$ must contain a nonzero number of integers with the same color $c_s$. Again, if our Weihrauch reduction is correct, $\lim_s c_s =c$ exists, and furthermore, there is a stage $t$ such that for all $s \geq t$, $c_s=c$. We can now complete the construction of $C$: at stage $s$, color $s$ with any color but $c_s$ (if $c_s$ exists). So $C$ cannot color infinitely many integers with color $c$. Contradiction.

We have that $\text{Rainbow Pigeonhole} <_{W} \text{``finite union of finite c.e.\ sets is finite''} <_{W} \mathrm{RT}^1_{< \infty}$ despite the fact that these three problems are logically equivalent. 

\section{Weihrauch reductions from $\mathrm{RT}^2_{< \infty}$ and $\mathrm{SRT}^2_{< \infty}$} \label{hirst3}

We now turn to Hirst's third result, Theorem~\ref{hirst}(c), showing that $\mathrm{RT}^2_{<\infty}$ implies $\mathrm{B}\Sigma^0_3$. This result was improved to $\mathrm{SRT}^2_{<\infty}$ in Cholak, Jockusch, and Slaman \cite{CJS-2001} (see Theorem~11.4) and also in Mytilinaios and Slaman \cite{Mytilinaios.Slaman:94}, Proposition~5.3. Again, our goal is to extract a Weihrauch reduction from that proof.  
  
First, we need an equivalent form of $\mathrm{B}\Sigma^0_3$ using $\mathrm{COF}$. Recall $\mathrm{COF} = \{ e \mid W_e \text{ is cofinite} \} = \{ e \mid \exists k \forall x> k \exists s (x \in W_{e,s})\}$ is $\Sigma^0_3$-complete.  
  
\begin{theorem}\label{BS3}
	The following are equivalent:
	\begin{enumerate}[label=(\alph*)]
		\item The intersection of finitely many cofinite c.e.\ sets is cofinite.
		\item $\mathrm{B}\Sigma^0_3$. 
	\end{enumerate}
\end{theorem}
 
\begin{proof}
	Let us formally write out (a):
	\begin{equation*}\begin{split}
		\forall i < n \exists b_i \bigg[ \forall x > b_i \exists s (x &\in W_{f(i),s}) \bigg]
		\implies \\ 
		\exists b \forall i < n  \bigg[ \forall x > b \exists s (x &\in W_{f(i),s})\bigg]
	\end{split}
	\end{equation*}
	The part in the brackets is $\Pi^0_2$, so this is an occurrence of $\mathrm{B}\Sigma^0_3$. So (b) implies (a). For the other direction, we will allow the reader to make the needed modifications in Equations~\ref{FU3}, \ref{BS2}, \ref{FU}, and \ref{FU2}.
\end{proof}
 
The contrapositive of (a) is: ``if the intersection of finitely many c.e.\ sets is not cofinite, then one of the c.e.\ sets is not cofinite''. The instance $I$ here is a collection of c.e.\ sets $\{W_{f(e)} \mid e < n\}$. Recall that a c.e.\ set need not be in $\mathcal{P}$; it is just $\Sigma^0_1$-definable over $\mathcal{A}$, possibly with parameters. Furthermore, if a c.e.\ set $W_{f(e)}$ is cofinite, we should not expect to find an $\mathcal{N}$-infinite set $X \in \mathcal{P}$ such that $X \cap W_e = \emptyset$. So a solution $S$ here is just an $e$ such that $W_{f(e)}$ is not cofinite. $\mathcal{A}(I,S)$ here just checks that if the intersection of the sets is not cofinite, then $W_{f(e)}$ is not cofinite. Let's call this problem $\mathrm{FICF}$ for the next theorem.  

\begin{theorem}\label{srt2i}
	$\mathrm{FICF}$ is Weihrauch reducible over $\mathcal{N}$ to $\mathrm{SRT}^2_{<\infty}$. 
\end{theorem}	

\begin{proof}
	Our instance $I_0$ is a collection of c.e.\ sets $\{W_{f(e)} \mid e < n \}$ whose intersection, $\bigcap\{W_{f(e)} \mid e < n \}$, is not cofinite. Given a pair $(x,s)$, if there is a least $y^{x,s}$ and then a least $e^{x,s}$ such that $x \leq y^{x,s} \leq s$ and $y^{x,s} \notin W_{f(e^{x,s}),s}$, let $\Phi(I_0)(x,s)=e^{x,s}$. Otherwise, let $\Phi(I_0)(x,s)=n$. For all $x$, there is a least $y^x \geq x$ and then a least $e^x<n$ where $y^x \notin W_{f(e^x)}$. Hence, $\lim_s y^{x,s} = y^x$ and $\lim_s e^{x,s} = e^x$. Therefore, this coloring is stable and the limit color is not $n$. Let $H_1$ be homogeneous for $\Phi(I_0)$. Let $H_1 = \{ h_0 < h_1 < \dots < h_i < h_{i+1} < \dots \}$. We need to determine $H_1$'s color, which is not necessarily given. By homogeneity of $H_1$, we have that $e^{h_i}= e^{h_i,h_{i+1}}$ is our solution to $I_0$, and using $I_0$ we can calculate $e^{h_i,h_{i+1}}$. 
\end{proof}

\section{Working with large $n$}

\subsection{$\mathrm{B}\Sigma^0_n$}\label{BSN}

In $\mathcal{N}$, $\overline{\emptyset^{(n)}}$ is $\Pi^0_{n}$-complete. Hence, every $\Sigma^0_{n+1}$-definable set can be considered as a $\Sigma^0_{1}$-definable set with an oracle or parameter for $\overline{\emptyset^{(n)}}$. When viewed in this manner, the first half of the proof of Theorem~\ref{equivBS2} relativizes to the following: 

\begin{theorem}
	In $\mathcal{N}$, the following are equivalent: for all $n$,
	\begin{enumerate}[label=(\alph*)]
		\item Every finite union of finite $\Sigma^0_{n+1}$-definable sets is finite.
		\item $\mathrm{B}\Sigma^0_{n+2}$.
	\end{enumerate}
\end{theorem}

We also want to relativize the second half of Theorem~\ref{equivBS2}. 

\begin{theorem}\label{useBn}
	Consider these statements in $\mathcal{N}$:
	\begin{enumerate}[label=(\alph*)]
		\item Every finite union of finite $\Sigma^0_{n+1}$-definable sets is finite.
		\item The Infinite $\Sigma^0_{n+1}$-Pigeonhole Principle holds. 
	\end{enumerate}
	(a) implies (b) and, if in addition, $\mathcal{N}$ realizes $\mathrm{B}\Sigma^0_{n}$, then (b) implies (a). 
\end{theorem}

The (a) implies (b) direction follows almost immediately.  Since likely we are looking for the first $n$ where $\mathrm{B}\Sigma^0_{n}$ fails in $\mathcal{N}$, this extra hypothesis can be tolerated. 

\begin{proof}
	(b) implies (a) follows: Like in the proof of Theorem~\ref{equivBS2}, it is enough to show that a finite union of finite $\Sigma^0_{n+1}$-definable sets induces a $\Sigma^0_{n+1}$ partial coloring where each color is used at most finitely often. Let $\{ x \mid \exists s \gamma_i(x,s) \}$, for $i < k$, be our $\Sigma^0_{n+1}$-definable finite sets, where $\gamma_i(x,s)$ is $\Pi^0_n$. $C(x)=i$ iff there is an $s$ such that $i$ is the least $e<k$ such that $\gamma_e(x,s)$ and, for all $t<s$ and all $i < k$, $\neg \gamma_i(x,t)$. We need $\mathrm{B}\Sigma^0_n$ to show this is equivalent to a $\Sigma^0_{n+1}$ formula in $\mathcal{N}$ (we need to move the bounded quantifiers in the last clause to the interior of the formula). 
\end{proof}

\subsection{$\mathrm{RT}^n_2$} \label{last}\label{RTN}

In this subsection, we assume that $\mathcal{N}$ realizes $\mathrm{RT}^{n+1}_2$. By Jockusch \cite{jockusch:72} and others, over $\mathrm{RCA}_0$, $\mathrm{RT}^{n+1}_2$ implies Arithmetic Comprehension ($\mathrm{ACA}_0$) when $n\geq 2$. This also holds over our fragment $F$. So $\mathcal{N}$ realizes arithmetic comprehension. Every $\Sigma^0_m$-definable set is in $\mathcal{P}$.   

To get arithmetic induction, it is enough to show every arithmetic set $X$ has a least element. By arithmetic comprehension, $X\in \mathcal{P}$. Applying $\mathrm{I}\Delta^0_0$ to the formula $x \not\in X$ shows that $X$ has a least element. Hence, $\mathcal{N}$ also realizes arithmetic bounding. 

We can also use this parameter trick to show that the Infinite Arithmetic Pigeonhole Principle holds in $\mathcal{N}$. Take our arithmetic coloring $C$. $C \in \mathcal{P}$. $C$ is now a $\Delta^0_0$ coloring (with a parameter for $C$), and Theorem~\ref{wr1} applies. 

There is also a Weihrauch reduction of the Infinite $\Sigma^0_{n-1}$-Pigeonhole Principle from $\mathrm{RT}^{n+1}_2$ without using this parameter trick. But for this, we have to alter the instances of our problem to include not only the partial $\Sigma^0_{n-1}$ coloring but also an $\mathcal{N}$-infinite subset $X$ of its domain. Only for $n=2$ can we trivially get hold of this set $X$. We also need that arithmetic bounding holds in $\mathcal{N}$. 

\begin{theorem}\label{lastlemma}
	The Infinite $\Sigma^0_n$-Pigeonhole Principle is (strongly) Weihrauch reducible over $\mathcal{N}$ to $\mathrm{RT}^{n+1}_2$ over $\mathcal{N}$. 
\end{theorem}

\begin{proof}
First we have to determine our functional $\Phi$. Let $C_0$ be a $\Sigma^0_{n-1}$ $k$-coloring such that $X$ is an $\mathcal{N}$-infinite subset of the domain of the coloring. Let $\exists w_0 \forall w_1 \dots \theta(x,c,\vec{w})$ be the $\Sigma^0_{n-1}$ graph of our coloring $C_0(x)=c$ (recall $\theta$ is always $\Delta^0_0$). Since we are working in $\mathcal{N}$, we can assume that our $\Sigma^0_{n-1}$ graph is in Prenex Normal Form. Given $x \in X$ and an $(n-1)$-tuple $\vec{z}$, let $c_{x,\vec{z}}$ be the least $c<k$ such that $\exists w_0 < z_0 \forall w_1 < z_1 \dots \theta(x,c,\vec{w})$ and, for all $c'< c$, $\forall w_0 < z_0 \exists w_1 < z_1 \dots \neg\theta(x,c',\vec{w})$. This is well-defined since $C_0(x)$ exists. Let $\Phi(C_0)(x,y,\vec{z}) = 1$ iff $c_{x,\vec{z}} = c_{y,\vec{z}}$, and $0$ otherwise, where $x,y,\vec{z} \in X$. This is a coloring of $[X]^{n+1}$. 

\begin{claim}
	Let $H_1 \subseteq X$ be homogeneous for $\Phi(C_0)$ and $D\subset H_1$ be finite. Then there is a $\vec{z} \in H_1$ such that for all $x\in D$, $C_0(x) = c_{x,\vec{z}}$. 
\end{claim}
   
\begin{proof}
	It is enough to show by induction that there is an $l$-tuple $\vec{z} \in H_1$ such that when $l$ is even, $\exists w_0 < z_0 \forall w_1 < z_1 \dots \forall w_{l-1} < z_{l-1} \exists w_{l} \forall w_{l+1} \dots \theta(x,C_0(x),\vec{w})$ and, for all $c'< C_0(x)$, $\forall w_0 < z_0 \exists w_1 < z_1 \dots \exists w_{l-1} < z_{l-1} \forall w_{l} \exists w_{l+1}\dots \neg\theta(x,c',\vec{w})$, and similarly for when $l$ is odd. To find $z_{l+1}$, we only need to worry about the existential quantifier. For each $w_0 < z_0, w_1 < z_1, \dots, w_l < z_l$, we have to find a bound for $w_{l+1}$. There are at most $z_0 z_1 \dots z_l$ such $w_{l+1}$. Since arithmetic bounding holds in $\mathcal{N}$, a bound for all these $w_{l+1}$ exists. 
\end{proof}

	Since the size of $D$ can be larger than the number of colors, it must be the case that $H_1$ has color $1$ and, for all $x,y \in H_1$, $C_0(x)=C_0(y)$. Again, if we want the color of the hole, the coloring $C_0$ is needed. 
\end{proof}

\section{Conclusion}

Perhaps we should always be asking when we can extract a Weihrauch reduction or similar reduction from a proof in $\mathcal{N}$ comparing two problems. There has been some recent work in this direction by Jeff Hirst and his students; see Davis, Hirst, Keohulian, Miller, and Ross \cite{MR5000368}. 

\bibliographystyle{plain}
\bibliography{BSigma2.bib}
\end{document}